\documentclass[a4paper,12pt,oneside,reqno]{amsart}
\usepackage{amssymb,amsfonts,amsmath,amsthm,amscd, bbm}
\usepackage{graphicx, color}
\numberwithin{equation}{section}  

\newcommand{\beq}{\begin{equation}} 
\newcommand{\eeq}{\end{equation}} 
\newcommand{\bea}{\begin{aligned}}
\newcommand{\eea}{\end{aligned}}
\newcommand{\bdm}{\begin{displaymath}}
\newcommand{\edm}{\end{displaymath}}
\newcommand{\barr}{\begin{array}}
\newcommand{\earr}{\end{array}}
\newcommand{\ben}{\begin{enumerate}}
\newcommand{\een}{\end{enumerate}}
\newcommand{\bde}{\begin{description}}
\newcommand{\ede}{\end{description}}

\newtheorem{teor}{Theorem}
\newtheorem{prop}[teor]{Proposition} 
\newtheorem{lem}[teor]{Lemma}

\newcommand{\R}{\mathbb{R}}

\newcommand{\PP}{\mathbb{P}}
\newcommand{\E}{{\mathbb{E}}}

\newcommand{\defi}{\equiv} 

\newcommand{\e}{\epsilon}

\begin{document}

\title[FPP in mean field]
{First passage percolation in the mean field limit} 

\author[N. Kistler]{Nicola Kistler}
\address{Nicola Kistler \\ J.W. Goethe-Universit\"at Frankfurt, Germany.}
\email{kistler@math.uni-frankfurt.de}

\author[A. Schertzer]{Adrien Schertzer}            
\address{adrien schertzer \\ J.W. Goethe-Universit\"at Frankfurt, Germany.}
\email{schertzer@math.uni-frankfurt.de}

\author[M. A. Schmidt]{Marius A. Schmidt}
\address{Marius A. Schmidt \\ J.W. Goethe-Universit\"at Frankfurt, Germany.}
\email{mschmidt@math.uni-frankfurt.de}

\thanks{It is a pleasure to thank Louis-Pierre Arguin and Olivier Zindy for important discussions in a 
preliminary phase of this work.}


 \date{\today}

\begin{abstract} 
The {\it Poisson clumping heuristic} has lead Aldous to conjecture the value of the first passage percolation on the hypercube in the limit of large dimensions. Aldous' conjecture has been rigorously confirmed by Fill and Pemantle [{\it Annals of Applied Probability} {\bf 3} (1993)] by means of a variance reduction trick. We present here a streamlined and, we believe, more natural proof based on ideas emerged in the study of Derrida's random energy models.
\end{abstract}

\maketitle

\section{Introduction}
We consider the following (oriented) first passage percolation (FPP) problem. Denote by 
$G_n = (V_n, E_n)$ the $n$-dimensional hypercube. $V_n = \{0,1\}^n$ is thus the set of vertices, and $E_n$ the set of edges connecting nearest neighbours. To each edge we attach independent, identically distributed random variables $\xi$. We will assume these to be mean-one exponentials.  (As will become clear in the treatment, this choice represents no loss of generality: only the behavior for small values matters).
We write $\boldsymbol{0}=(0,0, ... ,0)$ and ${\boldsymbol 1}=(1,1, ... ,1)$ for diametrically opposite vertices, and denote by  $\Pi_n$ the set of paths of length $n$ from \textbf {0} to \textbf {1}. Remark that $\sharp \Pi_n = n!$, and that
any $\pi \in$ $\Pi_n$ is of the form \textbf {0} = $\textit{v}_0$,$\textit{v}_1$, ..., $\textit{v}_n$ = $\textbf {1}$, 
with the $v's \in V_n$. To each path $\pi$ we assign its {\it weight} $$X_\pi\defi\sum_{(\textit{v}_j, \textit{v}_{j-1})\in\pi} \xi_{\textit{v}_{j-1},\textit{v}_j}.$$  
The FPP on the hypercube concerns the minimal weight
\beq \label{one}
m_n \defi \min_{\pi\in\Pi_n}X_\pi,
\eeq 
in the limit of large dimensions, i.e. as $n\to \infty$. The {\it leading} order has been conjectured by Aldous \cite{aldous}, and rigorously established by Fill and Pemantle \cite{Fill_Pemantle}:

\begin{teor}[Fill and Pemantle] \label{pf}
For the FPP on the hypercube,
\beq \label{pf_lim}
\lim_{n \to \infty} m_n = 1, 
\eeq
in probability. 
\end{teor}

The result is surprising, but then again not. On the one hand, it can be readily checked that \eqref{pf_lim} coincides with the large-$n$ minimum of $n!$ independent sums, each consisting of $n$ independent, mean-one exponentials. The FPP on the hypercube thus manages to reach  the same value as in the case of  {\it independent FPP}. In light of the severe correlations among the weights (eventually due to the tendency of paths to overlap), this is indeed a notable feat. On the other hand, the asymptotics involved is that of large dimensions, in which case (and perhaps according to some folklore) a {\it mean-field trivialization} is expected, in full agreement with Theorem \ref{pf}. The situation is thus reminiscent of Derrida's generalized random energy models, the GREMs \cite{rem, grem, kistler}, which are hierarchical Gaussian fields playing a fundamental role in the Parisi theory of mean field spin glasses. Indeed, for specific choice of the underlying parameters, the GREMs undergo a {\it REM-collapse} where the geometrical structure  is no longer detectable in the large volume limit, see also \cite{bolthausen_kistler, bovier_kurkova}. Mean field trivialization and REM-collapse are  two sides of the same coin.

The proof of Theorem \ref{pf} by Fill and Pemantle implements a {\it variance reduction trick} which is ingenious but, to our eyes, sligthly opaque. The purpose of the present notes is to provide a more natural proof which relies, first and foremost, on neatly exposing the aforementioned point of contact between the FPP on the hypercube and the GREMs. The key observation (already present in \cite{Fill_Pemantle}, albeit perhaps somewhat implicitly) is thereby the following well-known, loosely formulated property: 
\beq \bea \label{property}
&\text{\it in high-dimensional spaces, two walkers which depart from }\\
&\qquad \text{\it one another are unlikely to ever meet again.}
\eea \eeq
Underneath the FPP thus lies an {\it approximate  hierarchical structure}, whence the point of contact with the GREMs. Such a connection then allows to deploy the whole arsenal of mental pictures, insights  and tools recently emerged in the study of the {\it REM-class}: specifically, we use the multi-scale refinement of the 2nd moment method introduced in \cite{kistler}, a flexible tool which has proved useful in a variety of models,  most notably the log-correlated class, see e.g. \cite{arguin} and references therein. (It should be however emphasized that the FPP at hand is not, strictly speaking, a log-correlated field). 

Before addressing a model in the REM-class, it is advisable to first work out the details for the associated GREM, i.e. on a suitably constructed tree. In the specific case of the hypercube, 
one should rather think of two trees patched together, the vertices {\bf 0} and {\bf 1} representing the respective roots, see Figure \ref{fig_one} below.
\begin{figure}\label{fig_one}
\includegraphics[scale=0.31]{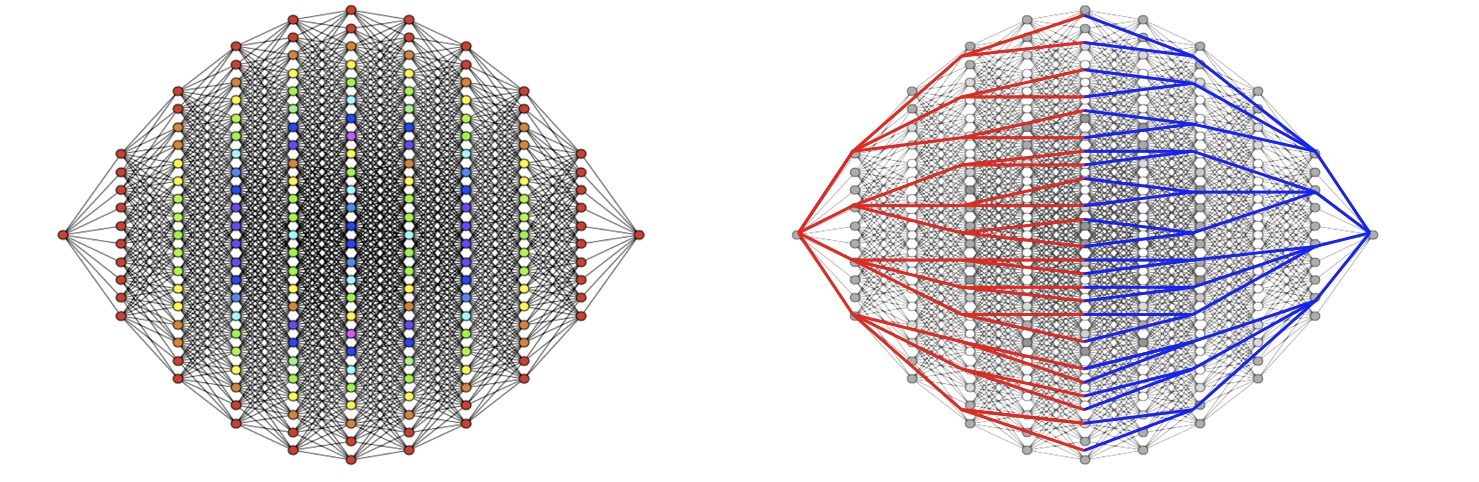}
\caption{A rendition of the 10-dim hypercube, and the associated trees patched together. Observe in particular how the branching factor decreases when wandering into the core of the hypercube: this is due to the fact that a walker starting out in {\bf 0} and heading to {\bf 1} has, after $k$ steps, $(N-k)$ possible choices for the next step. (The walker's steps correspond to the scales; the underlying trees are thus non-homogenuous, a fact already pointed out in \cite{aldous}). The figure should be taken {\it cum grano}: in case of the FFP, trees simply capture the aforementioned property of high-dimensional spaces, see \eqref{property} above, modulo the constraint that paths must start and end at prescribed vertices. .}
\end{figure}
For brevity, we restrain from giving the details for the tree(s), and tackle right away the FPP on the hypercube. Indeed, it will become clear below that once the connection with the GREMs is established, the problem on the hypercube reduces essentially to a delicate {\it path counting}, requiring in particular combinatorial estimates, many of which have however already been established in \cite{Fill_Pemantle}. 

The route taken in these notes neatly unravels, we believe, the physical mechanisms eventually responsible for the mean field trivialization. What is perhaps more, the point of contact with the REMs opens the gate towards some interesting and to date unsettled issues, such as the corrections to subleading order, or the weak limit. These aspects will be addressed elsewhere. \\

In the next section we sketch the main steps behind the new approach to Theorem \ref{pf}.
The proofs of all statements are given in a third and final section. 

\section{The multi-scale refinement of the 2nd moment method}
We will provide (asymptotically) matching lower and upper bounds for the FPP in the limit of large dimensions following the recipe laid out in \cite[Section 3.1.1]{kistler}.  The lower bound, which is the content
of the next Proposition, will follow seamlessly from Markov's inequality and some elementary path-counting. 
\begin{prop} \label{lower_bound_prop} 
For the FPP on the hypercube, 
\beq \label{lowerb}
\lim_{n \to \infty} m_n \geq 1,
\eeq
almost surely. 
\end{prop}

In order to state the main steps behind the upper bound, we need to introduce some additional notation. First, remark that the vertices of the $n$-hypercube stand in correspondence with the standard basis of $\R^n$. Indeed, every edge is parallel to some unit vector $e_j$, where $e_j$ connects $(0, \dots, 0)$ to $(0, \dots, 0, 1, 0, \dots, 0)$
with a $1$ in position $j$. We identify a path $\pi$ of length $n$ from $\boldsymbol 0$ to $\boldsymbol 1$ by a permutation of $1 2 \dots n$ say $\pi_1 \pi_2\dots \pi_n$. $\pi_l$ is giving the direction the path $\pi$ goes in step $l$, hence after $i$ steps the path $\pi_1 \pi_2\dots \pi_n$ is at vertex $\sum_{j\leq i} e_{\pi_j}$. We denote the edge traversed in the $i$-th step of $\pi$ by $[\pi]_i$ and define the weight of path $\pi$ by 
\[
X_\pi = \sum_{i\leq n} \xi_{[\pi]_i}\,
\]
where $\{\xi_{e}, e\in E_n\}$ are iid, mean-one exponentials and $T_n$ the space of permutations of $1 2 \dots n$ . Note that $[\pi]_i = [\pi']_j$ if and only if $i=j$, $\pi_i = \pi'_j$ and $\pi_1 \pi_2 \dots \pi_{i-1}$ is a permutation of $\pi'_1 \pi'_2 \dots \pi'_{j-1}$. \\

As mentioned, we will implement the multiscale refinement of the 2nd moment method from \cite{kistler}, albeit with a number of twists. In the multiscale refinement, the first step is "to give oneself an epsilon of room": we will indeed consider $\e >0$ and show that 
\beq \label{goal}
\lim_{n\to \infty} {\PP\left(\#\{\pi\in\textit{T}_n, \sum\limits_{i=1}^{n}{\xi_{[\pi]_i}}\leq1+\epsilon\}>0\right)}=1.
\eeq 
The natural attempt to prove the above via the Paley-Zygmund inequality is bound to 
fail due to the severe correlations. We bypass this obstacle   
partitioning the hypercube into three regions which we refer to as 'first', 'middle' and 'last', see Fig. \ref{fig_two} below, and handling on separate footings. (This step slightly differs from the recipe in \cite{kistler}).

\begin{figure}[h!]\label{fig_two}
\includegraphics[scale=0.2]{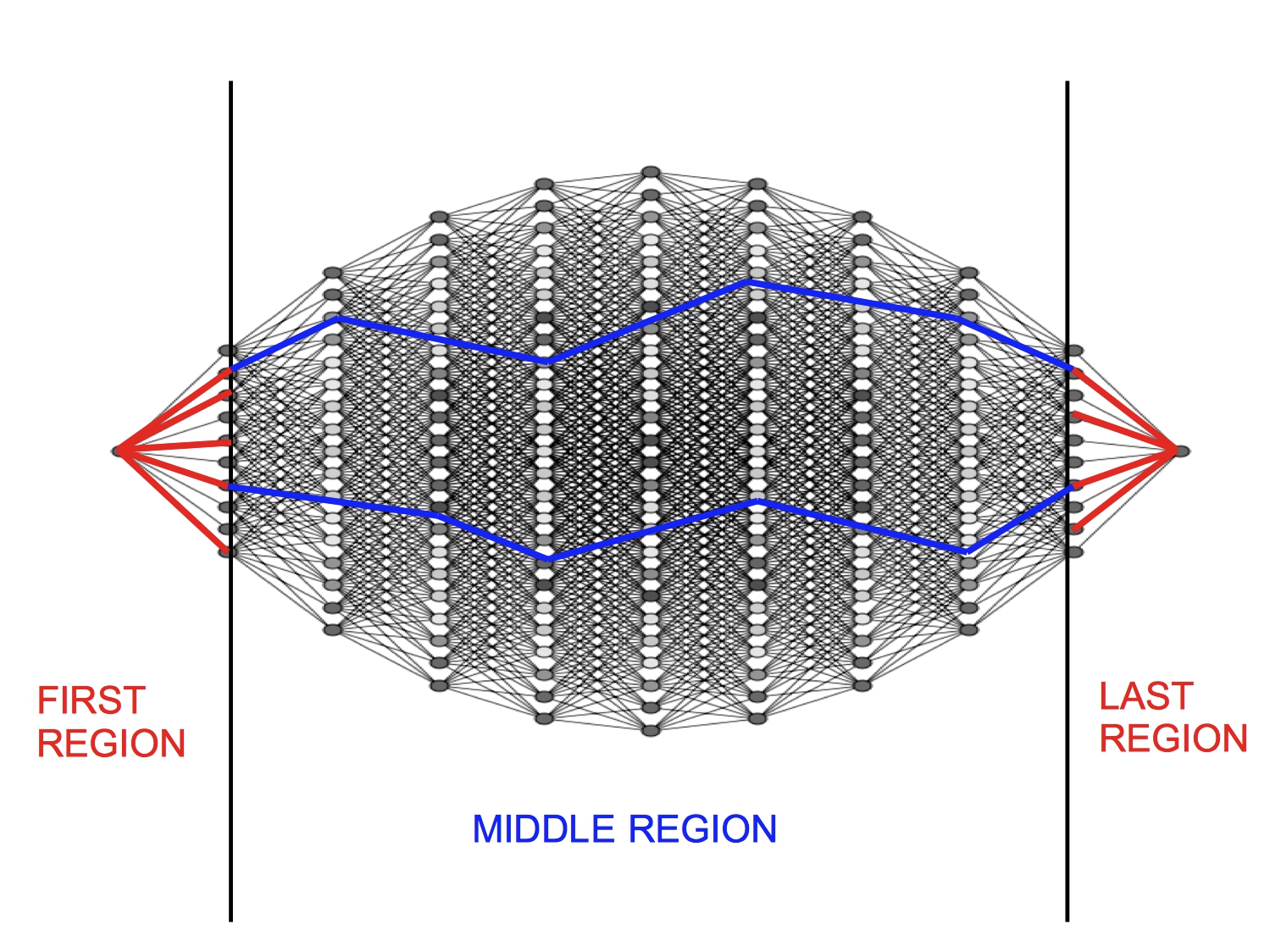}
\caption{Partitioning the hypercube into the three regions. Red edges are {\it $\e$-good}: their weight is smaller than $\epsilon/3$. Blue paths connecting first and last level have weights smaller than $1+\epsilon/3$. The total weight of a path consisting of one red edge outgoing from {\bf 0}, a connecting blue path, and a final red edge going into {\bf 1} is thus less than $1+\epsilon\;$. These are the relevant paths leading to  tight upper bounds for the FPP.} 
\end{figure}
We then address the first region, proving that one finds a {\it growing}  number of edges 
outgoing from $\boldsymbol 0$ with weight less than $\epsilon/3$. (By symmetry, the same then holds true for the last region). We will refer to these edges with low weights as {\it $\epsilon$-good}, or simply {\it good}. The existence of a positive fraction of good edges is the content of  Proposition \ref{countbeg} below.

\begin{prop} \label{countbeg}  
With
\beq \label{first and last} \bea
A^{\boldsymbol 0}_n & \defi \{v\leq n : (\boldsymbol 0,e_v)\in E_n \; \text{is $\e$-good} \} , \quad  A^{\boldsymbol 1}_n & \defi \{v\leq n: (\boldsymbol 1-e_v,\boldsymbol 1)\in E_n\; \text{is $\e$-good} \}\,,
\eea \eeq
there exists $C = C(\e) >0$ such that 
\beq
 \lim_{n\to \infty}{\PP\left(| A^{\boldsymbol 0}_n  \setminus A^{\boldsymbol 1}_n |\geq C n\right)}, \; {\PP\left(| A^{\boldsymbol 1}_n  \setminus A^{\boldsymbol 0}_n |\geq C n\right)}=1 \,.
\eeq
\end{prop}

\begin{proof}
Consider independent  exponentially (mean one) distributed random variables  $\{\xi_i\}, \{\xi_i'\}$. 
We have:
\beq\bea
&\frac{| A^{\boldsymbol 0}_n  \setminus A^{\boldsymbol 1}_n |}{n} \overset{d}{=} \frac{1}{n}\sum\limits_{i=1}^{n}{\mathbf{1}_{\{\xi_i \leq \frac{\epsilon}{3}, \xi'_i>\frac{\epsilon}{3}\}}}\underset{n \to \infty}{\overset{a.s.}{\longrightarrow}} p(\epsilon),
\eea\eeq
by the law of large numbers, where $p(\epsilon)=\PP(\xi_1 \leq \epsilon)\PP(\xi_1 > \epsilon)>0$. The claim thus holds true for any $C\in (0,p(\epsilon))$. The second claim is fully analogous.

\end{proof}
By the above, the missing ingredient in the proof of \eqref{goal} is thus the existence of (at least) one path in the middle region with weight less than $1+\e/3$, and which connects an $\e$-good edge in the first region to one in the last. This will be eventually done in Proposition \ref{connecting} by means of a full-fledged multiscale analysis. Towards this goal, consider the random variable accounting for {\it good paths} connecting $\boldsymbol 0$ and $\boldsymbol 1$ whilst going through good edges in first and last region, to wit:
\beq \label{connecting_rv}
\mathcal{N}_n=\#\left\{\pi\in\textit{T}_n: \pi_1\in  A^{\boldsymbol 0}_n  \setminus A^{\boldsymbol 1}_n, \pi_n \in  A^{\boldsymbol 1}_n  \setminus A^{\boldsymbol 0}_n  \; \text{and }\; \sum\limits_{i=2}^{n-1}{\xi_{[\pi]_i}}\leq1+\frac{\epsilon}{3} \right\}, 
\eeq
We now claim that
\beq \label{new_g}
\lim_{n\to \infty} \PP\left( \mathcal{N}_n>0 \right)  = 1,
\eeq
which would naturally imply \eqref{goal}. To establish \eqref{new_g}, we  exploit the existence of a wealth of good edges, 
\beq \label{new_g_2}
\PP\left(\mathcal{N}_n> 0\right) \geq  \PP\left(\mathcal{N}_n> 0, |A^{\boldsymbol 0}_n  \setminus A^{\boldsymbol 1}_n|\geq C n, |A^{\boldsymbol 1}_n  \setminus A^{\boldsymbol 0}_n|\geq C n \right)\eeq
Using that the weights involved in $A^{\boldsymbol 0}_n$ and $A^{\boldsymbol 1}_n $ are independent of all other weights and that considering more potential paths increases the probability of there beeing a path with specific properties we have that 
\[\PP\left(\mathcal{N}_n> 0 \mid |A^{\boldsymbol 0}_n  \setminus A^{\boldsymbol 1}_n| = j, |A^{\boldsymbol 1}_n  \setminus A^{\boldsymbol 0}_n|= k \right) \]
is monotonically growing in $j$ and $k$ as long as the probability is well defined, i.e. as long as $j+k\leq n$. Therefore
\[ (\ref{new_g_2}) \geq \PP\left(\mathcal{N}_n> 0 \mid |A^{\boldsymbol 0}_n  \setminus A^{\boldsymbol 1}_n| = \lceil Cn \rceil, |A^{\boldsymbol 1}_n  \setminus A^{\boldsymbol 0}_n|= \lceil Cn \rceil\right) \PP\left( |A^{\boldsymbol 0}_n  \setminus A^{\boldsymbol 1}_n| \geq Cn, |A^{\boldsymbol 1}_n  \setminus A^{\boldsymbol 0}_n| \geq Cn\right) \]
\[ =  \PP\left(\mathcal{N}_n> 0 \mid |A^{\boldsymbol 0}_n  \setminus A^{\boldsymbol 1}_n| = \lceil Cn \rceil, |A^{\boldsymbol 1}_n  \setminus A^{\boldsymbol 0}_n|=\lceil Cn\rceil \right) -o(1)\]
in virtue of Proposition \ref{countbeg} for properly chosen $C=C(\e)>0$. This in turn equals
\[ =  \PP\left(\mathcal{N}_n> 0 \mid A^{\boldsymbol 0}_n  \setminus A^{\boldsymbol 1}_n = A, A^{\boldsymbol 1}_n  \setminus A^{\boldsymbol 0}_n= A' \right) -o(1) \] 
for any admissible choice $A,A'$ with $|A| = |A'| = \lceil Cn \rceil$, say $A\defi \{j: j \leq Cn\}$ and $A' \defi \{j: j \geq (1-C)n\}$. Claim \eqref{new_g} will steadily follow from the 
\begin{prop} \label{connecting} (Connecting first and last region)
Let
\beq 
\textit{T}^{(1,n)}_n \defi \left\{\pi \in \textit{T}_n:\; \pi_1 \in A, \pi_n \in A' \right\}.
\eeq
It then holds: 
\[
\lim_{n\to \infty} \PP\left(\#\left\{\pi \in T^{(1,n)}_n: \; \sum_{i=2}^{n-1}{\xi_{[\pi]_i}}\leq1+ \epsilon/3\right\}> 0\right) = 1.
\]
\end{prop}
Since $\eqref{new_g}$ implies $\eqref{goal}$, the upper bound for the main theorem immediately follows from Propositions \ref{lower_bound_prop} and \ref{connecting}. It thus remains to provide the proofs of these two propositions: this is done in the next, and last section.
 
\section{Proofs}
\subsection{Tail estimates, and proof of the lower bound}
We first state some useful tail-estimates. 

\begin{lem} \label{law} Consider independent  exponentially (mean one) distributed random variables  $\{\xi_i\}, \{\xi_i'\}$. With $X_n \defi \sum_{i=1}^n \xi_i$ and $x>0$, it then holds:
\beq \label{law_i}
\PP\left( X_n\leq x\right) = \left(1+K(x,n) \right)\frac{e^{- x} x^n}{n!}, 
\eeq
with $0\leq K(x,n)\leq e^{ x} x/(n+1).$  \\

Given $X_n'\defi \sum_{i=1}^n \xi_i'$. Assume that $X_n'$ shares exactly $k$ \emph{edges} (meaning here k exponential random variables) with $X_n$, without loss of generality:
$$X_n' = \sum_{i=1}^k \xi_i + \sum_{i=k+1}^n \xi_i'.$$ Then
\beq  \label{law_ii}
\PP\left( X_n\leq x, X'_n\leq x\right) \leq \PP\left( X_n\leq x\right)\PP\left( X_{n-k}\leq x\right).
\eeq
\end{lem}
\begin{proof} One easily checks (say through characteristic functions) that $X_n$ is a $\text{Gamma}(n,1)$-distributed random variable, in which case
\beq\bea
\PP\left( X_n\leq x\right) &=\frac{1}{(n-1)!}\int\limits_{0}^x t^{n-1}e^{-t} dt = 1-e^{-x}\sum\limits_{k=0}^{n-1}\frac{x^k}{k!} \,,
\eea \eeq
the second step by partial integration.  We write the r.h.s. above as
\beq\bea
& e^{-x}\sum\limits_{k=n}^{\infty}\frac{x^k}{k!} = e^{-x}\frac{x^n}{n!}\left(1+\frac{n!}{x^n}\sum\limits_{k=n+1}^{\infty}\frac{x^k}{k!} \right).
\eea \eeq
By Taylor expansions, 
\beq\bea
\sum\limits_{k=n+1}^{\infty}\frac{x^k}{k!}\leq\frac{e^{x}x^{n+1}}{(n+1)!}, \\
\eea \eeq
hence \eqref{law_i} holds with 
\beq\bea
&K(x,n):=\frac{n!}{x^n}\sum\limits_{k=n+1}^{\infty}\frac{x^{k}}{k!}\leq \frac{e^{x}x}{(n+1)}\\
\eea \eeq
As for the second claim, by positivity of exponentials, 
\beq\bea
&\PP\left( X_n\leq x, X'_n\leq x\right) \leq \PP\left(\sum\limits_{i=1}^{n}{\xi_i}\leq x, \sum\limits_{i=k+1}^{n}{\xi'_i}\leq x\right)\,.
\eea\eeq
Claim \eqref{law_ii} thus follows from the independence of the $\xi, \xi'$ random variables.
\end{proof}

Armed with these estimates, we can move to the 

\begin{proof}[Proof of Proposition \ref{lower_bound_prop} (the lower bound)]
With  $\mathcal{N}_n^x=\#\{\pi \in \textit{T}_n, X_{\pi}\leq x\}$, it holds: 
\beq\bea \label{van}
\PP\left(m_n\leq x\right)=\PP\left(\mathcal{N}_n^x \geq 1\right) &\leq\E \mathcal{N}_n^x \\
&=n!\PP\left(X_\pi\leq x\right)\\
& \stackrel{ \eqref{law_ii}}{=} \left(1+o_n(1) \right)e^{- x}x^n\,,
\eea\eeq 
the second step by Markov inequality. Remark that \eqref{van} vanishes exponentially fast for any $x<1$; an elementary application of the Borel-Cantelli Lemma thus yields \eqref{lowerb} and "half of the theorem",  the {\it lower} bound,  is proven.  
\end{proof}

\subsection{Combinatorial estimates}
The proof of the upper bounds relies on a somewhat involved path-counting procedure. The required estimates are a variant of  \cite[Lemma 2.4]{Fill_Pemantle} and are provided by the following

\begin{lem}[Path counting] \label{path_counting} Let $\pi'$ be any reference path on the $n$-dim hypercube connecting $\boldsymbol{0}$ and
$\boldsymbol 1$, say $\pi'=12...n$. Denote by $f(n, k)$ the number of paths $\pi$ that share precisely $k$ edges ($k\geq1$) with $\pi'$ whithout considering the first and the last edge. Finally, shorten $\mathfrak{n_{e}} \defi n-5e(n+3)^{2/3}$.
\begin{itemize}
\item For any  $K(n) = o(n)$ as $n \to \infty$,
\beq \label{path_counting_i}
f(n,k) \leq (1+o(1))(k+1)(n-k-1)! \, 
\eeq 
uniformly in $k$ for $k\leq K(n).$
\item Suppose $k+2\leq \mathfrak{n_{e}}$. Then, for $n$ large enough,
\beq \label{path_counting_ii}
f(n,k)\leq 2n^6(n-k)! \,.
\eeq
\item Suppose $k \geq \mathfrak{n_{e}} -1$. Then, for $n$ large enough, 
\beq \label{path_counting_iii}
f(n,k)\leq \frac{1}{k!} (n-2)! (n-k-1) \,.
\eeq
\end{itemize}
\end{lem}

\begin{proof} [Proof of Lemma \ref{path_counting}]
To see  \eqref{path_counting_i}, consider a path $\pi$ which shares precisely $k$ edges with the reference path 
$\pi' =12\dots n$. We set $r_i = l$ if the $l$-th traversed edge by $\pi$ is the $i$-th shared edge of $\pi$ and $\pi'$. (We set by convention $r_0 =0$ and $r_k+1 = n+1$). Shorten $\textbf{r}\defi \textbf{r}(\pi)=(r_0,...,r_{k+1})$, and  $s_i \defi r_{i+1}-r_{i}$, i=0,...,k. For any sequence $\textbf{r}_0=(r_0,...,r_{k+1})$ with 0 = $r_0<r_1<...<r_k<r_{k+1}=n+1$, let C($\textbf{r}_0$) denote the number of path $\pi$ with $\textbf{r}(\pi)=\textbf{r}_0$.  Since the values $\pi_{r_i+1},...,\pi_{r_i+s_i-1}$ must be a permutation of $\{r_i+1,...,r_i+s_i-1\}$.
one easily sees that $C(\textbf{r})\leq G(\textbf{r})$, where 
\beq\label{G}
G(\textbf{r})=\prod\limits_{i=0}^{k}(s_i-1)! \,.
\eeq
Let now $j=\textit{j}(\textbf{r}) \defi \max_i(s_i-1)$. We will consider separately the cases $\textit{j}< n-4k $ and $\textit{j}\geq n-4k$, the underlying idea being that $G(\textbf{r})$ is small in the first case, and while not small in the second, there are only few sequences with such large $\textit{j}$-value.\\
Denote by $f_{\{j< n-4k\}}(n,k)$ resp. $f_{\{j\geq n-4k\}}(n,k)$ the number of paths $\pi$ that share precisely $k$ edges with $\pi'$  not counting the first and the last edge, where $j< n-4k$ for the first function and $j\geq n-4k$ for the second one. It holds:
\beq \label{decomp}
f(n,k)=f_{\{j< n-4k\}}(n,k)+f_{\{j\geq n-4k\}}(n,k) \,.
\eeq

\begin{itemize}
\item[] {\sf Case $j< n-4k$}. We claim that 
\beq 
G(\textbf{r}) \leq (n-4k-1)!(3k+1)! \,.
\eeq 
In fact, for $j\leq n-4k-1$, and by log-convexity,  the product in \eqref{G} is maximized at $\textbf{r}'s$ such that $\textit{j}(\textbf{r}) = n-4k-1$. It thus follows that 
\beq \bea \label{tot_1}
G(\textbf{r}) & \leq \left(\max_i(s_i-1)\right) ! \left( \sum\limits_{i}(s_i-1)-\max_i(s_i-1) \right)! 
\\ & \leq (n-4k-1)!(3k+1)! \,,
\eea \eeq
the last step since $\sum_{i}(s_i-1)=n-k$. On the other hand, the number of  $\textbf{r}$-sequences under consideration is at most $\binom {n-2}{k}$: combining with \eqref{tot_1}, 
\beq\bea\label{tot_2}
&f_{\{j< n-4k\}}(n,k)\leq \frac{(n-4k-1)!(3k+1)!(n-2)!}{(k)!(n-2-k)!} \\
& =(n-k-1)!\frac{(n-4k-1)!}{(n-k-1)!}\frac{(n-2)!}{(n-2-k)!}\frac{(3k+1)!}{(k)!}\\
&\leq2(k+1)(n-k-1)!\left[\frac{(3k)^2(n-2)}{(n-4k)^3}\right]^k \,,
\eea\eeq
by simple bounds.  The term in square brackets converges to 0 as n $\to \infty$ uniformly in $k$ as long as $k\leq K(n)=o(n)$, hence the contribution from the first case is $o((k+1)(n-k-1)!)$, uniformly in such $k's$. \\

\item[] {\sf Case $j \geq n-4k$}.   Again by log-convexity of factorials,  
\beq 
G(\textbf{r}) \leq j(\textbf{r})!(n-k-j(\textbf{r}))! \,.
\eeq
The number of $\textbf{r}$-sequences for which $j(\textbf{r}) = j_0$ is at most $(k+1)$ times the number of $\textbf{r}$-sequences with $s_0-1=j_0$; since the $k-1$ common edges have to be placed before the last edge in our definition of $f(n,k)$, the latter is thus at most $\binom{n-1-j_0-1}{k-1}$. For fixed $j_0$,
the contribution is therefore at most 
\beq\bea \label{tot_3}
&\frac{(n-1-j_0-1)!(k+1)j_0!(n-j_0-k)!}{(k-1)!(n-1-j_0-k)!} \\
& \hspace{3cm} =\frac{(n-j_0-2)!(k+1)j_0!(n-j_0-k)}{(k-1)!}\,.
\eea\eeq
Summing \eqref{tot_3} over all possible values $n-4k \leq j_0 \leq  n-k-1$, we get  
\beq\bea \label{tot_4}
& f_{\{j\geq n-4k\}}(n,k) \leq (k+1)(n-k-1)!\sum\limits_{j_0=n-4k}^{n-k -1}\frac{(n-j_0-2)!j_0!(n-j_0-k)}{(k-1)!(n-k-1)!}\\
&\leq(k+1)(n-k-1)!\sum\limits_{i=1}^{3K(n)} i \left(\frac{n}{4K(n)}-1\right)^{1-i} \\
&\leq(k+1)(n-k-1)! \left(1+\sum\limits_{i=1}^{3K(n)} 2^i \left(\frac{n}{4K(n)}-1 \right)^{-i}  \right)\\
&=(k+1)(n-k-1)!(1+o_n(1)) \,.
\eea\eeq
Using the upperbounds \eqref{tot_2} and \eqref{tot_4} in  \eqref{decomp} settles the proof of  \eqref{path_counting_i}. \\
\end{itemize} 

The second claim of the Lemma relies on estimates established by Fill and Pemantle, which we now recall. Denote by $f_1(n,k)$ the number of paths $\pi$ that share precisely $k$ edges with the reference path $\pi' = 12\cdots n$. (Contrary to $f(n,k)$,  first and  last edge do matter here!) By \cite[Lemma 2.4]{Fill_Pemantle} the following holds
\beq \label{useful_pf}
f_1(n,k)\leq n^6(n-k)! \,,
\eeq
as soon as $k\leq \mathfrak{n_e}$ and $n$  is large enough.
It then holds: 
\beq \bea
f(n,k) & \leq f_1(n,k)+f_1(n,k+1)+f_1(n,k+2) \\
& \stackrel{\eqref{useful_pf}}{\leq} n^6(n-k)!(1+\frac{1}{(n-k)}+\frac{1}{(n-k)(n-k-1)})\\
&\leq  2 n^6(n-k)!\,,
\eea \eeq
yielding \eqref{path_counting_ii}. \\

It remains to address the third claim of the Lemma, which we recall reads
\beq \label{path_counting_third}
f(n,k)\leq \frac{1}{k!} (n-2)! (n-k-1)! \,,
\eeq
for $\mathfrak{n_{e}}-1\leq k \leq  n$.  For this, it is enough to proceed by worst-case:  there are at most $(n-k-1)!$ paths sharing $k$ edges with the reference-path $\pi'$ for given \textbf{r}, and $\binom {n-2}{k}$ ways to choose such  \textbf{r}-sequences. All in all, this leads to
\beq\bea
f(n,k)\leq \binom {n-2}{k}(n-k-1)!=\frac{(n-2)!(n-k-1)}{k!} \,,
\eea\eeq
settling the proof of \eqref{path_counting_third}.
\end{proof}

\subsection{Proof of the upper bound}

\begin{proof}[Proof of Proposition \ref{connecting} (Connecting first and last region)] The claim is that 
\beq \label{uu_zero}
\lim_{n\to \infty} \PP\left( {\mathcal{N}^{(1)}_n}>0\right) = 1,
\eeq
where $\mathcal{N}^{(1)}_n=\#\{\pi\in \textit{T}^{(1,n)}_n, \sum_{i=2}^{n-1}{\xi_{[\pi]_i}}\leq1+\frac{\epsilon}{3}\}$. This will now follow from the Paley-Zygmund inequality, which requires control of 1st- and 2nd-moment estimates. As for the 1st moment, by simple counting and with $C$ as in Proposition \ref{countbeg},
\beq\bea \label{last_first}
\E \mathcal{N}^{(1)}_n &= C^2 n^2 (n-2)!\times \PP\left(\sum_{i=2}^{n-1}{\xi_{[\pi]_i}}\leq1+\frac{\epsilon}{3} \right) \\
& = \kappa n^2 \left(1+\frac{\epsilon}{3}\right)^{n-2} [1+ o(1)] \qquad(n\to \infty), 
\eea\eeq
(the last step by Lemma \ref{law}) for some numerical constant $\kappa>0$.

Now shorten $A\defi \{\pi,\pi'' \in T_n $ {\it have no edges in common in the middle region}$ \}$.  For the 2nd moment, it holds: 
\beq\bea\label{2ndC}
\E\left[ {\mathcal{N}^{(1)}_n}^{2} \right] &=\sum\limits_{(\pi,\pi'') \in A}\PP\left(\sum\limits_{i=2}^{n-1}{\xi_{[\pi]_i}}\leq1+\frac{\epsilon}{3}\right)^2\\
&\qquad\qquad +\sum\limits_{(\pi,\pi'') \in {A}^c}\PP\left(\sum\limits_{i=2}^{n-1}{\xi_{[\pi]_i}}\leq1+\frac{\epsilon}{3}, \sum\limits_{i=2}^{n-1}{\xi_{[\pi'']_i}}\leq1+\frac{\epsilon}{3}\right) \\
& =:   (\Sigma_{A}) + (\Sigma_{A^c}),  \quad \text{say}.
\eea\eeq
But by independence, 
\beq\bea
(\Sigma_{A}) \leq \left(\E \mathcal{N}^{(1)}_n \right)^2,
\eea\eeq
hence it steadily follows from \eqref{2ndC} that
\beq\bea \label{bound_l_r}
1 \leq \frac{\E\left[ {\mathcal{N}^{(1)}_n}^{2} \right]}{\left(\E \mathcal{N}^{(1)}_n \right)^{2}}\leq 1 + \frac{(\Sigma_{A^c})}{\left(\E \mathcal{N}^{(1)}_n \right)^2}.
\eea\eeq
It thus remains to prove that 
\beq \label{control comp}
(\Sigma_{A^c}) = o\left( \E\left[\mathcal{N}^{(1)}_n \right]^2 \right) \qquad  (n\to \infty)\,.
\eeq
To see \eqref{control comp}, by symmetry it suffices to consider the case where $\pi''$ is any reference path, say $\pi''=\pi' = 12 \cdots n$.  By the second claim of Lemma \ref{path_counting}, and with  $X_{n}$ denoting a $\text{Gamma}(n,1)$-distributed random variable, it holds: 
\beq\bea
(\Sigma_{A^c}) &\leq   (Cn)^2 (n-2)!\sum\limits_{k=1}^{n-3}{f(n,k)\PP\left( X_{n-2}\leq 1+\frac{\epsilon}{3}\right)\PP\left( X_{n-2-k}\leq 1+\frac{\epsilon}{3}\right)}\\
&\qquad + (Cn)^2 (n-2)! \PP\left( X_{n-2}\leq 1+\frac{\epsilon}{3}\right) \,,
\eea\eeq
hence 
 \beq\bea \label{maj_1}
\frac{(\Sigma_{A^c})}{\left(\E \mathcal{N}^{(1)}_n \right)^2} &\leq \frac{1}{(Cn)^2(n-2)!}\left(\sum\limits_{k=1}^{n-3}{f(n,k)\frac{\PP\left( X_{n-2-k}\leq 1+\frac{\epsilon}{3}\right)}{\PP\left( X_{n-2}\leq 1+\frac{\epsilon}{3}\right)}+\frac{1}{\PP\left( X_{n-2}\leq 1+\frac{\epsilon}{3}\right)}}\right)\\
\eea\eeq
By Lemma \ref{law}, 
\beq\bea \label{maj_2}
\frac{\PP\left( X_{n-2-k}\leq 1+\frac{\epsilon}{3}\right)}{\PP\left( X_{n-2}\leq 1+\frac{\epsilon}{3}\right)}\leq \frac{2(n-2)!}{(n-k-2)!(1+\frac{\epsilon}{3})^k}\,,
\eea\eeq
and therefore, up to the irrelevant $o(1)$-term,  
\beq \bea \label{maj_2}
\eqref{maj_1} &\leq \frac{2}{(Cn)^2} \sum\limits_{k=1}^{n-3}\frac{f(n,k)}{(n-k-2)!(1+\frac{\epsilon}{3})^k}\, \\
& = \frac{2}{(Cn)^2} \left( \sum_{k=1}^{K(n)} + \sum_{k=K(n)+1}^{\mathfrak{n_{e}}-2 }+ \sum_{k=\mathfrak{n_{e}} -1}^{n-3} \right) \frac{f(n,k)}{(n-k-2)!(1+\frac{\epsilon}{3})^k}\,, 
\eea \eeq
where $K(n) \defi n^{1/4}$  and  $\mathfrak{n_{e}} = n-5e(n+3)^{2/3}$. By Lemma \ref{path_counting} the first sum on the r.h.s. of \eqref{maj_2} is at most
\beq \bea \label{maj_2_bis}
& \frac{2}{(Cn)^2} \sum_{k=1}^{n^{1/4}}\frac{f(n,k) }{(n-k-2)!(1+\frac{\epsilon}{3})^k} \\
& \stackrel{\eqref{path_counting_i}}{\leq}  \frac{2}{(Cn)^2} \sum_{k=1}^{n^{1/4}}  \frac{2 (k+1)(n-k-1)!}{(n-k-2)!(1+\frac{\epsilon}{3})^k} \\
& \leq \frac{4(n^{1/4}+1)}{C^2n}\sum\limits_{k=1}^{n^{1/4}}{\frac{1}{(1+\frac{\epsilon}{3})^k}} =  \frac{12}{C^2\epsilon} n^{-3/4}[1+o(1)]\underset{n \to \infty}{\longrightarrow} 0.
\eea \eeq
As for the second sum on the r.h.s. of \eqref{maj_1},
\beq \bea \label{maj_3}
& \frac{2}{(Cn)^2} \sum_{k=n^{1/4}+1}^{{\mathfrak n_e}-2} \frac{ f(n,k)}{(n-k-2)!(1+\frac{\epsilon}{3})^k} \\
& \stackrel{\eqref{path_counting_ii}}{\leq}   \frac{  4n^6  }{(Cn)^2} \sum\limits_{k=n^{1/4}+1}^{\mathfrak{n_{e}}-2}{\frac{(n-k)!}{(n-k-2)!(1+\frac{\epsilon}{3})^k}}\\
&\leq  \frac{  4n^6  }{C^2} \sum\limits_{k=n^{1/4}+1}^{\mathfrak{n_{e}}-2} \left(1+\frac{\epsilon}{3} \right)^{-k} \leq \frac{12  n^6}{\e C^2}   \left(1+\frac{\epsilon}{3} \right)^{-n^{1/4}} [1+o(1)],
\eea \eeq
which is thus also vanishing in the large-$n$ limit.  It thus remains to check that the same is true for the third and last term on the r.h.s. of \eqref{maj_2}:
\beq \bea \label{maj_4}
& \frac{2}{(Cn)^2} \sum_{{\mathfrak n_e}-1}^{ n-3}\frac{f(n,k)}{(n-k-2)!(1+\frac{\epsilon}{3})^k} \\
& \stackrel{\eqref{path_counting_ii}}{\leq}    \frac{2}{(Cn)^2} \sum\limits_{k=\mathfrak{n_e}-1}^{n-3}{\frac{(n-k-1)}{k!}\frac{(n-2)!}{(n-2-k)!(1+\frac{\epsilon}{3})^k}}\\
&\leq   \frac{ 2}{C^2n} \sum\limits_{k=\mathfrak{n_e}-1}^{n-3}  \binom {n-2}{\mathfrak{n_e}-1}\left(1+\frac{\epsilon}{3} \right)^{-k}\,,
\eea \eeq
the last inequality by simple estimates on the binomial coefficients (using $\mathfrak{n_e}-1\geq n/2$). 
Remark that 
\beq \bea \label{maj_5}
\eqref{maj_4} & \leq\frac{6}{\epsilon C^2n}\binom {n-2}{\mathfrak{n_e}}{(1+\frac{\epsilon}{3})}^{2-\mathfrak{n_e}}\\
\eea \eeq
By Stirling's formula, one plainly checks that 
\beq\bea
&\binom {n-2}{\mathfrak{n_e}}\leq \frac{n!}{\mathfrak{n_e}!} =\frac{n^ne^{\mathfrak{n_e}-n}}{(\mathfrak{n_e})^\mathfrak{n_e}}[1+o(1)]\,.\\
\eea\eeq
Plugging this estimate into \eqref{maj_5} we thus get for some numerical constant $\kappa>0$ that
\beq\bea \label{maj_5}
\eqref{maj_4} &\leq \kappa \frac{n^n}{\left(\mathfrak{n_e}\left(1+\frac{\epsilon}{3} \right)\right)^\mathfrak{n_e}} \underset{n \to \infty}{\longrightarrow} 0,
\eea\eeq
and  \eqref{control comp} follows. An elementary application of the Paley-Zygmund inequality then settles the proof  of Proposition \ref{connecting}.
\end{proof}

\end{document}